\documentclass[12pt]{article}
\usepackage[a4paper,margin=1in]{geometry}
\usepackage{authblk}
\usepackage{amsmath,amssymb,amsthm,mathtools}
\usepackage{mathrsfs}
\usepackage{enumitem}
\usepackage[hidelinks]{hyperref}

\newcommand{\cA}{\mathcal A}
\newcommand{\cC}{\mathcal C}
\newcommand{\cI}{\mathcal I}
\newcommand{\cP}{\mathscr P}
\newcommand{\cR}{\mathcal R}
\newcommand{\HH}{\mathscr H}
\newcommand{\EE}{\mathbb E}
\newcommand{\Hom}{\operatorname{Hom}}
\newcommand{\Ext}{\operatorname{Ext}}
\newcommand{\Aut}{\operatorname{Aut}}
\newcommand{\Ker}{\operatorname{Ker}}
\newcommand{\Coker}{\operatorname{Coker}}
\newcommand{\proj}{\operatorname{proj}}
\newcommand{\modu}{\operatorname{mod}}
\newcommand{\add}{\operatorname{add}}
\newcommand{\Ktwo}{K^{[-1,0]}}

\theoremstyle{plain}
\newtheorem{theorem}{Theorem}[section]
\newtheorem{proposition}[theorem]{Proposition}
\newtheorem{lemma}[theorem]{Lemma}
\newtheorem{corollary}[theorem]{Corollary}
\theoremstyle{remark}
\newtheorem{remark}[theorem]{Remark}

\title{Quantum Weyl Relations arising from Two-Term Complexes}

\author[1]{Qinghua Chen
}
\author[2]{Yonggang Hu\thanks{Corresponding author. Email: \texttt{huyonggang@cfau.edu.cn}}}

\affil[1]{School of Mathematics and Statistics, Fuzhou University,
Fuzhou 350108, China}
\affil[2]{Department of International Economics, China Foreign Affairs University,
Beijing 100037, China}
\date{}

\begin{document}
\maketitle

\begin{abstract}
Let $Q$ be a Dynkin quiver over $k=\mathbb F_q$, let $A=kQ$, and let
$\Ktwo(\cP)$ be the extriangulated category of two-term complexes of
projective $A$-modules. We study the square-root normalized Hall algebra of
$\Ktwo(\cP)$. We first establish a PBW-type vector-space factorization into
the Ringel--Hall part and the shifted-projective part, and derive an explicit
mixed multiplication formula. For the indecomposable projectives $P_i$, let
$p_i$ and $z_i$ be the normalized Hall classes of $(0\to P_i)$ and
$(P_i\to0)$, respectively. We prove
\[
 z_ip_i=q^{-1}p_iz_i+1,
\]
and determine all off-diagonal products $z_ip_j$ in terms of kernels and
cokernels of maps $P_i\to P_j$. Hence each pair $(p_i,z_i)$ generates a
rank-one quantum Weyl algebra, while every pairwise Hom-orthogonal family of
projectives generates a higher-rank quantum Weyl subalgebra. For these
subalgebras we construct explicit Hall--Fock modules, on which the $p_i$ act
as creation operators and the $z_i$ act as $q$-annihilation operators. This
provides a finite-field Hall model parallel to categorical Hall-type Weyl
actions in Donaldson--Thomas theory. Finally, we show that BGP reflection
transports the corresponding reflection subalgebras and preserves the mixed
Hall coefficients.
\end{abstract}



\section{Introduction}

Hall algebras turn extension data into multiplication. Ringel's realization
of $U_v^+(\mathfrak g_Q)$ from the Hall algebra of a Dynkin quiver initiated
the modern theory \cite{R90}; Green's bialgebra structure and Ringel's PBW
bases made the quantum-group structure explicit \cite{Gr95,R96}. Hall
algebras of complexes subsequently produced constructions beyond ordinary
module categories, most notably Bridgeland's realization of the whole
quantum group from two-periodic complexes \cite{B13}.

The category $\Ktwo(\cP)$ of two-term projective complexes is a natural
extriangulated category \cite{NP19,Gar24}. It is closely related to support
$\tau$-tilting theory \cite{AIR14}, and its Hall algebra is covered by the
Hall theory of extriangulated categories developed by Wang, Wei, and Zhang
\cite{WWZ22}. The aim of this paper is not to give a full presentation of
that Hall algebra. Instead, we isolate the Weyl-type relations already
visible among the projective boundary objects, construct their natural
Fock-type representations, and explain how the relations depend on the
orientation of $Q$.

For each $A$-module $M$, let $C_M$ be its minimal projective presentation,
and for each projective $P$, put $Z_P=(P\to0)$. Every two-term projective
complex is uniquely isomorphic to $C_M\oplus Z_P$. In the normalized Hall
basis, multiplication therefore gives a PBW-type factorization
\[
 \cR\otimes\cI\xrightarrow{\sim}\HH(\Ktwo(\cP)),
 \qquad c_M\otimes z_P\longmapsto c_Mz_P,
\]
as vector spaces. Here $\cR$ is the Ringel--Hall part and $\cI$ is the
shifted-projective part. We emphasize that this is a two-factor ordered
factorization, not a triangular decomposition in the usual three-part
sense.

The mixed product $z_Pc_M$ is governed by maps $P\to M$. The categorical
reason is particularly transparent: if $f:P\to M$ corresponds to an
$\mathbb E$-extension of $Z_P$ by $C_M$, then its middle term is
\[
 C_{\Coker f}\oplus Z_{\Ker f}.
\]
Thus the Hall product records the kernel--cokernel geometry of a module
homomorphism. Specializing to indecomposable projectives gives
\[
 z_ip_i=q^{-1}p_iz_i+1.
\]
For different vertices, a nonzero map $P_i\to P_j$ contributes the Hall
class of its cokernel. The orientation dependence is therefore produced by
the projective geometry of the quiver rather than imposed externally.

Weyl-type commutators also arise in geometric Hall constructions. Toda
constructed Hall-type algebra structures on categorical Donaldson--Thomas
theories for local surfaces and actions on Pandharipande--Thomas categories.
The corresponding Hecke operators provide creation and annihilation
operators whose commutator on $K$-theory is analogous to a Weyl algebra
relation \cite{Toda20}. Our construction takes place in a different,
finitary setting, but it exhibits the same broad mechanism: Hall-type
correspondences produce creation, annihilation, and a correction term. To
make this parallel precise on the representation-theoretic side, we
construct Hall--Fock modules for every projectively Hom-orthogonal family of
vertices. In these modules, $p_i$ acts by multiplication and $z_i$ by a
$q^{-1}$-difference operator.

Our first main theorem combines the PBW-type factorization, the explicit
mixed formula, and the Weyl consequences. The Hall--Fock theorem then turns
the Weyl subalgebras into explicit operator representations. A final theorem
describes BGP reflection. At a sink or source, the reflection functor induces
an isomorphism between the corresponding reflection subalgebras; its
restriction to the Ringel--Hall part agrees with Lusztig's braid symmetry
\cite{BGP73,XY01,DX02}.

A companion paper studies a genuinely different problem: in the
unnormalized twisted Hall basis, two Ringel--Hall embeddings are used to give
a complete presentation and to construct automorphisms and
anti-automorphisms. The present paper is instead concerned with explicit
normalized Hall coefficients, Weyl substructures, Fock-type modules, and
orientation covariance.
\section{Two-term complexes and normalized Hall multiplication}

Let $Q$ be a Dynkin quiver with vertices $1,\ldots,n$, let $k=\mathbb F_q$,
and put
\[
 A=kQ,\qquad \cA=A\text{-}\modu,\qquad \cP=\proj A,
 \qquad \cC=\Ktwo(\cP).
\]
For $M\in\cA$, fix a minimal projective resolution
\[
 0\longrightarrow P_M^{-1}\xrightarrow{d_M}P_M^0
 \longrightarrow M\longrightarrow0
\]
and set $C_M=(P_M^{-1}\xrightarrow{d_M}P_M^0)$. For $P\in\cP$, set
\[
 C_P=(0\longrightarrow P),\qquad Z_P=(P\longrightarrow0).
\]
The category $\cC$ carries the standard extriangulated structure induced
from $K^b(\cP)$, with
\[
 \EE(X,Y)=\Hom_{K^b(\cP)}(X,Y[1]).
\]

\begin{proposition}\label{prop:decomposition}
Every object of $\cC$ is uniquely isomorphic to $C_M\oplus Z_P$, with
$M\in\cA$ and $P\in\cP$. Moreover,
\begin{align*}
 \Hom_{\cC}(C_M,C_N)&\cong\Hom_A(M,N),&
 \EE(C_M,C_N)&\cong\Ext_A^1(M,N),\\
 \EE(Z_P,C_M)&\cong\Hom_A(P,M),&
 \EE(C_M,Z_P)&=0.
\end{align*}
The remaining Hom and negative-extension spaces are obtained by viewing
$C_M$ as $M$ and $Z_P$ as $P[1]$ in $D^b(\cA)$.
\end{proposition}

\begin{proof}
Split the contractible summands of a two-term projective complex. Since $A$
is hereditary, the degree $-1$ homology is projective, which yields the
stated decomposition. The Hom and extension formulas follow from the
standard equivalence between projective resolutions and objects of
$D^b(\cA)$; see \cite{Gar24}.
\end{proof}

For $X\in\cC$, write $a_X=|\Aut_{\cC}(X)|$ and
\[
 \{X,Y\}=\prod_{r>0}|\EE^{-r}(X,Y)|^{(-1)^r}.
\]
The normalized Hall algebra $\HH(\cC)$ has basis $v_X$, indexed by
isomorphism classes, and multiplication
\[
 v_Xv_Y=\sum_{[L]}
 \frac{|\EE(X,Y)_L|}{|\Hom_{\cC}(X,Y)|}
 \left(\frac{a_L}{a_Xa_Y}\right)^{1/2}
 \frac{1}{\{X,Y\}}v_L.
\]
Set
\[
 c_M=v_{C_M},\qquad z_P=v_{Z_P},\qquad
 p_i=c_{P_i},\qquad z_i=z_{P_i}.
\]
Let
\[
 \cR=\operatorname{span}_{\mathbb C}\{c_M\mid M\in\cA\},\qquad
 \cI=\operatorname{span}_{\mathbb C}\{z_P\mid P\in\cP\}.
\]

\begin{proposition}[PBW-type factorization]\label{prop:factorization}
Multiplication induces a vector-space isomorphism
\[
 \cR\otimes\cI\xrightarrow{\sim}\HH(\cC),
 \qquad c_M\otimes z_P\longmapsto c_Mz_P.
\]
After fixing a convex order $\beta_1<\cdots<\beta_N$ on the positive roots,
the elements
\[
 E_{\beta_1}^{a_1}\cdots E_{\beta_N}^{a_N}
 z_1^{b_1}\cdots z_n^{b_n}
\]
form a basis of $\HH(\cC)$.
\end{proposition}

\begin{proof}
The subspace $\cR$ is, after a diagonal change of basis, the twisted
Ringel--Hall algebra, so Ringel's PBW theorem applies. Extensions of an
object $C_M$ by an injective object $Z_P$ split, and $c_Mz_P$ is a nonzero
scalar multiple of $v_{C_M\oplus Z_P}$. The result follows from
Proposition~\ref{prop:decomposition}.
\end{proof}

\section{Mixed multiplication and quantum Weyl subalgebras}

For $P\in\cP$ and $M\in\cA$, define
\[
 {}_Q\Hom_A(P,M)_L=
 \{f:P\to M\mid\Ker f\cong Q,\ \Coker f\cong L\}.
\]

\begin{lemma}[Middle terms of mixed extensions]\label{lem:middle-term}
Under the natural isomorphism
\[
 \EE(Z_P,C_M)\cong\Hom_A(P,M),
\]
let $\delta_f$ correspond to $f:P\to M$. If
\[
 Q=\Ker f,\qquad L=\Coker f,
\]
then every realization of $\delta_f$ has middle term isomorphic to
\[
 C_L\oplus Z_Q.
\]
Consequently,
\[
 |\EE(Z_P,C_M)_{C_L\oplus Z_Q}|
 =|{}_Q\Hom_A(P,M)_L|.
\]
\end{lemma}

\begin{proof}
Let
\[
 0\longrightarrow P_M^{-1}\xrightarrow{d_M}P_M^0
 \xrightarrow{\pi_M}M\longrightarrow0
\]
be the fixed minimal projective resolution of $M$. Choose a lifting
$\widetilde f:P\to P_M^0$ satisfying $\pi_M\widetilde f=f$. The extension
corresponding to $f$ is represented by the associated morphism
$Z_P\to C_M[1]$, and its middle term is
\[
 X_f\simeq\operatorname{Cone}(\widetilde f)[-1]
 =\left(P\oplus P_M^{-1}
 \xrightarrow{(\widetilde f,d_M)}P_M^0\right).
\]
The exact sequence
\[
 0\longrightarrow Q\longrightarrow P\xrightarrow{f}M
 \longrightarrow L\longrightarrow0
\]
and the definition of $X_f$ give
\[
 H^{-1}(X_f)\cong Q,
 \qquad
 H^0(X_f)\cong L.
\]
Since $A$ is hereditary, Proposition~\ref{prop:decomposition} yields
\[
 X_f\cong C_{H^0(X_f)}\oplus Z_{H^{-1}(X_f)}
 \cong C_L\oplus Z_Q.
\]
The final equality follows from the bijection
$\EE(Z_P,C_M)\cong\Hom_A(P,M)$.
\end{proof}

\begin{lemma}[Mixed multiplication]\label{lem:mixed}
For $P\in\cP$ and $M\in\cA$,
\[
 z_Pc_M=
 \sum_{[Q],[L]}
 |{}_Q\Hom_A(P,M)_L|
 \frac{|\Ext_A^1(L,Q)|}{|\Hom_A(L,Q)|}
 \left(\frac{a_La_Q}{a_Pa_M}\right)^{1/2}
 c_Lz_Q.
\]
\end{lemma}

\begin{proof}
By Lemma~\ref{lem:middle-term}, the extensions with middle term
$C_L\oplus Z_Q$ are precisely those arising from maps $P\to M$ with kernel
$Q$ and cokernel $L$. Substituting this count, together with the Hom,
negative-extension, and automorphism factors from
Proposition~\ref{prop:decomposition}, into the normalized Hall product gives
the formula.
\end{proof}

Put
\[
 h_{ij}=\dim_k\Hom_A(P_i,P_j).
\]
Because the underlying Dynkin graph is a tree, $h_{ij}\in\{0,1\}$, and for
$i\ne j$ at most one of $h_{ij},h_{ji}$ is nonzero.

\begin{theorem}[Main theorem]\label{thm:main}
The normalized Hall algebra has the PBW-type factorization of
Proposition~\ref{prop:factorization}, and the projective pairs $p_i,z_i$
satisfy the following relations.
\begin{enumerate}[label=\rm(\alph*)]
\item For all $i,j$,
\[
 p_ip_j=q^{h_{ji}-h_{ij}}p_jp_i,
 \qquad
 z_iz_j=q^{h_{ji}-h_{ij}}z_jz_i.
\]
\item For every $i$,
\[
 z_ip_i=q^{-1}p_iz_i+1.
\]
\item If $i\ne j$ and $h_{ij}=0$, then
\[
 z_ip_j=q^{-h_{ji}}p_jz_i.
\]
If $h_{ij}=1$, choose a nonzero map $P_i\to P_j$, let $L_{ij}$ be its
cokernel, and put
\[
 \gamma_{ij}=(q-1)
 \left(\frac{a_{L_{ij}}}{a_{P_i}a_{P_j}}\right)^{1/2}.
\]
Then
\[
 z_ip_j=p_jz_i+\gamma_{ij}c_{L_{ij}}.
\]
\item For every $i$, the subalgebra $\mathbb C\langle p_i,z_i\rangle$ is
isomorphic to
\[
 A_1(q)=\mathbb C\langle x,\partial\rangle/
 (\partial x-q^{-1}x\partial-1).
\]
More generally, if $I\subseteq Q_0$ is pairwise projectively
Hom-orthogonal, then
\[
 \mathbb C\langle p_i,z_i\mid i\in I\rangle
 \cong\bigotimes_{i\in I}A_1(q).
\]
\end{enumerate}
\end{theorem}

\begin{proof}
The relations among the $p_i$ and among the $z_i$ follow by evaluating the
split Hall products in both orders. For $z_ip_j$, apply
Lemma~\ref{lem:mixed}. If $i=j$, the zero endomorphism gives the split term
and the $q-1$ invertible endomorphisms have contractible middle term, giving
$z_ip_i=q^{-1}p_iz_i+1$. If $i\ne j$ and $h_{ij}=0$, only the zero map
$P_i\to P_j$ occurs, and its coefficient is $q^{-h_{ji}}$. If $h_{ij}=1$,
all nonzero maps are scalar multiples of one injective map and have the same
cokernel $L_{ij}$, yielding the correction term.

The ordered monomials $p_i^az_i^b$ occur among the basis elements of
Proposition~\ref{prop:factorization}, so the rank-one Weyl homomorphism is
injective. The same basis argument proves the Hom-orthogonal higher-rank
statement.
\end{proof}

\begin{remark}
Theorem~\ref{thm:main} does not identify the entire Hall algebra with the
standard $n$-th quantum Weyl algebra. In a non-semisimple Dynkin type, the
additional positive-root Hall elements in the PBW basis give more directions
than the $2n$ standard Weyl variables. The theorem identifies canonical Weyl
subalgebras and the quiver correction terms between them.
\end{remark}

\section{Hall--Fock modules and creation--annihilation operators}

Let $I\subseteq Q_0$ be projectively Hom-orthogonal, that is,
\[
 \Hom_A(P_i,P_j)=\Hom_A(P_j,P_i)=0
 \qquad(i\ne j,\ i,j\in I).
\]
By Theorem~\ref{thm:main}, the subalgebra
\[
 \mathcal W_I(q)=\mathbb C\langle p_i,z_i\mid i\in I\rangle
\]
is isomorphic to the tensor product of the rank-one quantum Weyl algebras
$A_1(q)$. Define the left ideal
\[
 \mathcal J_I=\sum_{i\in I}\mathcal W_I(q)z_i
\]
and the cyclic left module
\[
 \mathcal F_I=\mathcal W_I(q)/\mathcal J_I.
\]
Let $|0\rangle=1+\mathcal J_I$ be its vacuum vector. For
$\mathbf m=(m_i)_{i\in I}\in\mathbb N^I$, put
\[
 |\mathbf m\rangle=\prod_{i\in I}p_i^{m_i}|0\rangle.
\]
The product is independent of the chosen order because the $p_i$, $i\in I$,
commute.

\begin{theorem}[Hall--Fock representation]\label{thm:fock}
The following statements hold.
\begin{enumerate}[label=\rm(\alph*)]
\item The vectors $|\mathbf m\rangle$, $\mathbf m\in\mathbb N^I$, form a
basis of $\mathcal F_I$.
\item The elements $p_i$ act as creation operators:
\[
 p_i|\mathbf m\rangle=|\mathbf m+\mathbf e_i\rangle.
\]
\item The elements $z_i$ act as $q$-annihilation operators:
\[
 z_i|\mathbf m\rangle
 =[m_i]_{q^{-1}}|\mathbf m-\mathbf e_i\rangle,
\]
where
\[
 [m]_{q^{-1}}=1+q^{-1}+\cdots+q^{-(m-1)}
 =\frac{1-q^{-m}}{1-q^{-1}},
\]
and the right-hand side is zero when $m_i=0$.
\end{enumerate}
\end{theorem}

\begin{proof}
The PBW basis of $\mathcal W_I(q)$ consists of the ordered monomials
\[
 \prod_{i\in I}p_i^{a_i}\prod_{i\in I}z_i^{b_i}.
\]
Modulo $\mathcal J_I$, precisely the monomials with all $b_i=0$ survive,
which proves (a). Part (b) is immediate. The relation
\[
 z_ip_i=q^{-1}p_iz_i+1
\]
implies, by induction on $m$,
\[
 z_ip_i^m=q^{-m}p_i^mz_i+[m]_{q^{-1}}p_i^{m-1}.
\]
For $j\ne i$ in $I$, one has $z_ip_j=p_jz_i$. Applying the displayed
identity to the vacuum vector, which is annihilated by every $z_i$, proves
(c).
\end{proof}

\begin{corollary}[Polynomial realization]\label{cor:polynomial}
There is a vector-space isomorphism
\[
 \mathcal F_I\xrightarrow{\sim}\mathbb C[t_i\mid i\in I],
 \qquad
 |\mathbf m\rangle\longmapsto\prod_{i\in I}t_i^{m_i},
\]
under which $p_i$ acts by multiplication by $t_i$, while $z_i$ acts by the
Jackson $q^{-1}$-difference operator
\[
 D_{q^{-1},i}f
 =\frac{f(t_1,\ldots,t_i,\ldots)-
 f(t_1,\ldots,q^{-1}t_i,\ldots)}{(1-q^{-1})t_i}.
\]
\end{corollary}

\begin{proof}
The operator $D_{q^{-1},i}$ sends $t_i^m$ to
$[m]_{q^{-1}}t_i^{m-1}$ and satisfies
\[
 D_{q^{-1},i}t_i=q^{-1}t_iD_{q^{-1},i}+1.
\]
The assertion follows from Theorem~\ref{thm:fock}.
\end{proof}

\begin{proposition}\label{prop:simple-fock}
For every vertex $i$, the rank-one Hall--Fock module
\[
 \mathcal F_i=
 \mathbb C\langle p_i,z_i\rangle/
 \mathbb C\langle p_i,z_i\rangle z_i
\]
is simple.
\end{proposition}

\begin{proof}
Let $0\ne f(p_i)|0\rangle$ lie in a nonzero submodule and let $d$ be the
degree of $f$. Since $q>1$, all quantum integers $[m]_{q^{-1}}$ are
nonzero. Repeated application of $z_i$ shows that the submodule contains a
nonzero scalar multiple of $|0\rangle$. Applying powers of $p_i$ then
generates all of $\mathcal F_i$.
\end{proof}

\subsection*{Comparison with categorical Hall-type Weyl actions}

Toda constructed Hall-type algebra structures on categorical
Donaldson--Thomas theories for local surfaces and actions of categorical
Hall algebras of zero-dimensional sheaves on Pandharipande--Thomas
categories \cite{Toda20}. The associated Hecke correspondences give
creation and annihilation operators on $K$-groups, and their commutator is
expressed by tensor multiplication with a universal geometric class. Toda
interprets this relation as a categorification of a Weyl algebra action.

The present construction is not a specialization of Toda's geometric
setting. Our operators are elements of a finitary Hall algebra over a finite
field, and their diagonal relation has the scalar form
\[
 z_ip_i-q^{-1}p_iz_i=1.
\]
Nevertheless, the two constructions exhibit a common Hall-theoretic
pattern: a creation correspondence, an annihilation correspondence, and a
Weyl-type correction term. In the two-term category, the correction term is
produced by the invertible endomorphisms of $P_i$; in Toda's setting, it is
encoded by a universal class on a moduli space of stable pairs. The
Hall--Fock module above gives an explicit finite-field representation model
for this broader creation--annihilation phenomenon.

\section{BGP reflection and orientation covariance}

Let $i$ be a sink of $Q$, put $Q'=\sigma_iQ$, and use primes for objects
associated with $Q'$. Let
\[
 \Sigma_i^+:\cA_i^+(Q)\xrightarrow{\sim}\cA_i^-(Q')
\]
be the usual BGP equivalence between the subcategories having no direct
summand $S_i$ and $S_i'$, respectively. Put
\[
 \cP_i(Q)=\add\{P_j\mid j\ne i\}.
\]
Define
\[
 \HH_i^+(Q)=\operatorname{span}\{c_Mz_P\mid
 M\in\cA_i^+(Q),\ P\in\cP_i(Q)\},
\]
and define $\HH_i^-(Q')$ analogously.

\begin{proposition}
[Reflection property]\label{thm:reflection}
The rule
\[
 \mathfrak R_i^+(c_Mz_P)=
 c_{\Sigma_i^+M}'z_{\Sigma_i^+P}'
\]
defines an algebra isomorphism
\[
 \mathfrak R_i^+:\HH_i^+(Q)\xrightarrow{\sim}\HH_i^-(Q').
\]
Its inverse is induced by the negative reflection functor. For $j\ne i$,
\[
 \mathfrak R_i^+(p_j)=p_j',\qquad
 \mathfrak R_i^+(z_j)=z_j',
\]
and the rank-one Weyl relation at $j$ is preserved. The mixed coefficients in
Lemma~\ref{lem:mixed} are invariant under reflection. On the Ringel--Hall
part, $\mathfrak R_i^+$ agrees with the corresponding restriction of
Lusztig's braid symmetry.
    
\end{proposition}

\begin{proof}
The restricted BGP functor is an exact equivalence and preserves Hom spaces,
Ext spaces, automorphism groups, kernels, cokernels, and extension middle
terms. Hence it preserves every factor in the normalized Hall coefficient
and in Lemma~\ref{lem:mixed}. It sends $P_j$ to $P_j'$ for $j\ne i$, which
gives the formulas on the projective Weyl pairs. The identification with
Lusztig's symmetry is the classical Hall interpretation of BGP reflection
\cite{XY01,DX02}.
\end{proof}

\begin{remark}
The reflected vertex is excluded because the ordinary BGP functor kills the
exceptional simple there. A full braid-group action requires a double or
modified Hall enlargement.
\end{remark}

\section*{Acknowledgement}
The first author is supported by the Natural Science Foundation of Fujian Province
(Grant No. 2024J01361). 
The second author is supported by National Natural Science Foundation of China (Grant No.  12401049).


\begin{thebibliography}{99}
\bibitem{AIR14}
T. Adachi, O. Iyama and I. Reiten,
\textit{$\tau$-tilting theory},
Compos. Math. \textbf{150} (2014), 415--452.

\bibitem{BGP73}
I. N. Bernstein, I. M. Gel'fand and V. A. Ponomarev,
\textit{Coxeter functors and Gabriel's theorem},
Russian Math. Surveys \textbf{28} (1973), 17--32.

\bibitem{B13}
T. Bridgeland,
\textit{Quantum groups via Hall algebras of complexes},
Ann. of Math. \textbf{177} (2013), 739--759.

\bibitem{DX02}
B. Deng and J. Xiao,
\textit{Ringel--Hall algebras and Lusztig's symmetries},
J. Algebra \textbf{255} (2002), 357--372.

\bibitem{Gar24}
M. Garcia,
\textit{On algebraic and geometric aspects of the category of projective presentations},
Ph.D. thesis, Universit\'e Paris-Saclay, 2024.

\bibitem{Gr95}
J. A. Green,
\textit{Hall algebras, hereditary algebras and quantum groups},
Invent. Math. \textbf{120} (1995), 361--377.

\bibitem{NP19}
H. Nakaoka and Y. Palu,
\textit{Extriangulated categories, Hovey twin cotorsion pairs and model structures},
Cah. Topol. G\'eom. Diff\'er. Cat\'eg. \textbf{60} (2019), 117--193.

\bibitem{R90}
C. M. Ringel,
\textit{Hall algebras and quantum groups},
Invent. Math. \textbf{101} (1990), 583--591.

\bibitem{R96}
C. M. Ringel,
\textit{PBW-bases of quantum groups},
J. Reine Angew. Math. \textbf{470} (1996), 51--88.

\bibitem{WWZ22}
L. Wang, J. Wei and H. Zhang,
\textit{Hall algebras of extriangulated categories},
J. Algebra \textbf{610} (2022), 366--390.

\bibitem{Toda20}
Y. Toda,
\textit{Hall-type algebras for categorical Donaldson--Thomas theories on local surfaces},
Selecta Math. (N.S.) \textbf{26} (2020), Paper No.~64, 72 pp.

\bibitem{XY01}
J. Xiao and S. Yang,
\textit{BGP-reflection functors and Lusztig's symmetries: a Ringel--Hall algebra approach to quantum groups},
J. Algebra \textbf{241} (2001), 204--246.
\end{thebibliography}
\end{document}